\documentclass[a4paper]{amsart}
\usepackage{latexsym,amsxtra,amscd,ifthen}
\usepackage{amsfonts}
\usepackage{pifont}
\usepackage{verbatim}
\usepackage{amsmath}
\usepackage{graphicx}
\usepackage{amsthm}
\usepackage{amssymb}
\usepackage{url,color}
\usepackage[arrow,matrix]{xy}
\usepackage{pict2e,epic,eepic,pstricks}
\usepackage{bbm}

\newtheorem{thm}{Theorem}[section]

\newtheorem{cor}[thm]{Corollary}
\newtheorem{lem}[thm]{Lemma}

\newtheorem{prop}[thm]{Proposition}
\newtheorem{prop-def}[thm]{Proposition-Definition}
\theoremstyle{definition}
\newtheorem{Def}[thm]{Definition}
\theoremstyle{remark}
\newtheorem{rmk}[thm]{\bf Remark}
\newtheorem{ex}[thm]{\bf Example}
\numberwithin{equation}{section}
\def\la{\lambda}
\def\op{{\rm op}}
\def\P{\mathcal{P}}
\def\Q{\mathcal{Q}}
\def\U{\mathcal{U}}
\def\K{\Bbbk}
\def\ot{\otimes}

\def\al{\alpha}

\def\Ext{{\rm Ext}}
\def\xr{\xrightarrow}

\def\HP{{\rm HP}}
\def\H{{\rm H}}
\def\HL{{\rm HL}}

\def\PDer{{\rm PDer}}
\def\HH{{\rm HH}}

\def\Hom{{\rm Hom}}

\def\Tor{\mathrm{Tor}}

\def\vp{\varphi}
\def\id{{\rm id}}

\def\Ker{{\rm Ker}}
\def\Der{{\rm Der}}
\def \1{\mathbbm 1}

\def \I {{\rm I}}
\def\II {{\rm II}}
\def\uoQ {{\ot}_{\Q}}
\def\ft {\ltimes_f}
\def\gt{\ltimes_g}
\def\IPDer{{\rm IPDer}}
\def\sl {{\mathfrak{sl}}}

\begin{document}
\title[Cohomology structures for a Poisson algebra: II]
{Cohomology structures for a Poisson algebra: II}

\author[Y.-H. Bao and Y.Ye]
{Yan-Hong Bao\quad and Yu Ye$^*$}

\thanks{$^*$ The corresponding author}
\subjclass[2010]{16E40, 16W10, 17B37}
\date{\today}
\thanks{E-mail:
baoyh$\symbol{64}$ahu.edu.cn (Y.-H. Bao), yeyu@ustc.edu.cn (Y. Ye)}
\keywords{Poisson algebra, Poisson cohomology, formal deformation, deformation quantization}

\maketitle

\dedicatory{}%
\commby{}%

\begin{abstract}
For a Poisson algebra,
we prove that the Poisson cohomology theory introduced in [M. Flato, M. Gerstenhaber and A. A. Voronov, Cohomology and
Deformation of Leibniz Pairs, Lett. Math. Phys. 34 (1995) 77--90] is given by a certain derived functor. We show that the (generalized) deformation quantization is equivalent to the formal deformation for Poisson algebras under certain mild conditions. Finally we construct a long exact sequence, and use it to calculate the Poisson cohomology groups via the Yoneda-extension groups of certain quasi-Poisson modules and the Lie algebra cohomology groups.
\end{abstract}

\section{Introduction}

M. Flato, M. Gerstenhaber and A. A. Voronov developed a cohomology theory and a formal deformation theory for Poisson algebras
(not necessarily commutative), see \cite{FGV}. They showed that this cohomology controls those formal deformations such that the associative multiplication and the Lie bracket are simultaneously deformed. We call this cohomology FGV-Poisson cohomology, or simply Poisson cohomology.

An immediate question is whether the FGV-Poisson cohomology is exactly given by usual Yoneda-extensions or derived functors.
We introduce a complex $\chi_\bullet(A)$ of free Poisson modules for a Poisson algebra $A$, and show that the FGV-Poisson cohomology is exactly defined by the derived functor associated to the this complex, which gives an affirmative answer to the above question.

\begin{thm} [Theorem \ref{PC by der func}]
 Let $A$ be a Poisson algebra and $M$ a Poisson module over $A$. Then
\[\HP^n(A, M)\cong \H^n \Hom_\P(\chi_\bullet(A), M).\]
\end{thm}

An interesting observation is that Kontsevich's deformation quantization can be interpreted as a certain formal deformation of Poisson algebras. As a consequence, we may give a necessary condition for the existence of deformation quantization by using the FGV-Poisson cohomology group. To be precise, a Poisson algebra has no quantization deformations whenever the second FGV-Poisson cohomology group vanishes.

\begin{prop}[Proposition \ref{condition for dq and fd}]
Let $P=(A, \cdot, \{-,-\})$ be a nontrivial commutative Poisson algebra. If $\HP^2(A)=0$, then $P$ has no
deformation quantization.
\end{prop}

A well-known result by Farkas and Letzter says that a prime noncommutative algebra admits only standard Poisson structures(\cite[Theorem 1.2]{FL}). In some sense, this result will lead to a problem of lack of (non-standard) noncommutative Poisson algebras, and for this reason more general Poisson structures have been introduced and studied \cite{Xu, RVO}. While in the study of deformation quantization, it will have an unexpected application. It guarantees the existence of (generalized) deformation quantization under certain condition.

\begin{thm}[Theorem \ref{dq and fd}] Let $P=(A, \cdot, \{-,-\})$ be a commutative Poisson algebra.

$(1)$ If $P$ admits an $n$-deformation quantization for some $n\ge 1$, then $P$ can be deformed to a standard Poisson algebra.

$(2)$ Assume further that each formal deformation $(A[[t]], m_t)$
of $(A, \cdot)$ has only standard Poisson structures.
Then $P$ has an $n$-deformation quantization for some positive integer $n$ if and only if $P$ has a formal deformation.
\end{thm}

In particular, as it is shown in Example \ref{ex-intdom}, if a Poisson algebra is an integral domain as an associative algebra, then it has a (generalized) deformation quantization if and only if it has nontrivial deformation.

Recall that for a commutative Poisson algebra, Lichnerowicz has introduced in \cite{Li} a cohomology theory, which we call Lichnerowicz-Poisson cohomology (LP-cohomology for short). We would like to mention that for commutative Poisson algebras, the FGV-Poisson cohomology is different from the LP-cohomology. Roughly speaking, the FGV-Poisson cohomology controls the formal deformations which deform the product and the Lie bracket simultaneously, while the LP-cohomology controls those ones which only deform the Lie bracket. In Section 5, we compare the FGV-Poisson cohomology groups and the LP-chomology groups in lower degrees for commutative Poisson algebras.

The FGV-Poisson cohomology group is quite nontrivial to calculate in general. However, we may construct a long exact sequence, involving the FGV-Poisson cohomology groups, the Yoneda-extension groups of certain quasi-Poisson modules and the Lie algebra cohomology groups, see Theorem \ref{l.e.s. of Poisson modules}. This enables us to calculate the FGV-cohomology via the Lie algebra cohomology and the quasi-Poisson cohomology, where the former has been well studied by many authors, and the latter was discussed in our previous paper \cite{BY}.

The paper is organized as follows. In section 2, we recall some
basics. In Section 3 we construct a bicomplex
of free Poisson modules for a Poisson algebra, whose total complex applies to compute the FGV-Poisson cohomology groups.
Section 4 explains how Kontsevich's deformation quantization is related to formal deformation of Poisson algebras, and some conditions of the existence of (generalized) deformation quantization are also provided there in the language of the FGV-Poisson cohomology. In Section 5, we compare the FGV-Poisson cohomology  with the LP-cohomology for commutative Poisson algebras.
In Section 6, we show a long exact sequence, and apply to calculate the FGV-Poisson cohomology groups by using Lie algebra cohomology and quasi-Poisson cohomology. In the last section, we study the standard Poisson algebra of the $2\times 2$ matrix algebra. We compute its Poisson cohomology groups of lower degrees. Moreover, we show that in this case, any Poisson 2-cocycle lifts to a formal deformation.

%

Throughout, $\K$ will be a field of characteristic zero, all associative algebras over $\K$ have a
multiplicative identity element, and all unadorned $\Hom$ and $\otimes$ will be $\Hom_\K$ and $\otimes_\K$, respectively.
For simplicity, we denote by $A^i$ and $\wedge^j$ the tensor product and
the $j$-th exterior power of the $\K$-space $A$, respectively.

\section{Preliminaries}

A triple $(A,\cdot,\{-,-\})$ is called a \emph{Poisson algebra} over $\K$,
if $(A,\cdot)$ is an associative $\K$-algebra (not necessarily commutative),
$(A,\{-,-\})$ is a Lie algebra over $\K$, and the Leibniz rule
$\{ab,c\}=a\{b,c\}+\{a,c\}b$ holds for all $a,b,c\in A$.
A \emph{quasi-Poisson $A$-module} $M$ is both an $A$-$A$-bimodule and a Lie module
over $(A,\{-,-\})$ with the action given by $\{-,-\}_\ast: A\times M\to M$, which satisfies
\begin{align*}
\{a,bm\}_\ast &= \{a,b\}m+b\{a,m\}_\ast,\\
\{a,mb\}_\ast &= m\{a,b\}+\{a,m\}_\ast b
\end{align*}
for all $a,b\in A$ and $m\in M$. In addition, if
\[\{ab,m\}_\ast=a\{b,m\}_\ast+\{a,m\}_\ast b\]
holds for all $a,b\in A$ and $m\in M$, then $M$ is called a \emph{Poisson $A$-module}. Let $M, N$ be (quasi-)Poisson modules.
A homomorphism of (quasi-)Poisson $A$-modules
is a $\K$-linear function $f: M\to N$ which is
a homomorphism of both $A$-$A$-bimodules and Lie modules.

The following convention is handy in calculation, and we refer to \cite{YYY} for more details.

Denote by $A^\op$ the opposite algebra of the associative algebra $A$.
To avoid confusion, we usually use $a$ to denote an element in $A$
and $a'$ its corresponding element in $A^\op$.
Denote by  $A^e$ the \emph{enveloping algebra} $A\ot A^\op$ of the associative algebra $(A, \cdot)$,
and by $\U(A)$ the \emph{universal enveloping algebra} of the Lie algebra $(A, \{-,-\})$.
It is well known that the category of $A$-$A$-bimodules is isomorphic to the $A^e$-module category
and the category of Lie modules over $(A, \{-, -\})$ is isomorphic to the $\U(A)$-module category.
Note that $\U(A)$ is a cocommutative Hopf algebra with the comultiplication $\Delta(\al)=\sum \al_1 \otimes \al_2$ for any $\al\in \U(A)$
and the counit map $\epsilon\colon \U(A)\to \K$, and $A^e$ is a $\U(A)$-module algebra
with the action given by
\[\al(a\ot b')=\sum \al_1(a)\otimes (\al_2(b))'\]
for $\al\in \U(A), a\ot b'\in A^e$, where $\al(a)=\{x_1, \{x_2, \cdots, \{x_n, a\} \cdots \}\}$
for $\al=x_1\ot x_2\ot\cdots \ot x_n\in \U(A)$ and  $a\in A$.

%
%

\begin{Def} {\rm (\cite{YYY})}
Let $A$ be a Poisson algebra. The smash product
$A^e\# \U(A)$ is called the \emph{quasi-Poisson enveloping algebra}
of $A$ and denoted by $\Q$. The \emph{Poisson enveloping algebra} of
$A$, denoted by $\P$, is defined as the quotient algebra
$\Q/J$, where $J$ is the ideal of $\Q$ generated by $1_A\ot
1'_A\#(ab)-a\ot 1'_A\# b -1_A\ot b' \# a$ for all $a,b\in A$.
\end{Def}

\begin{thm}{\rm (\cite{YYY})} The category of quasi-Poisson modules over $A$
is isomorphic to the category of left $\Q$-modules, and
the category of Poisson modules over $A$ is isomorphic to the category of left $\P$-modules.
\end{thm}

Given a quasi-Poisson $A$-module $M$, one can define a $\Q$-module $M$ by setting
$(a\ot b' \#\al)m =a\al(m)b$
for all $m\in M$ and $a\ot b' \#\al\in \Q$. Conversely,
given a left $\Q$-module $M$, we set
$am=(a\ot 1'_A\# \1)m$, $ma=(1_A\ot a'\#\1)m$, and $\{a, m\}_\ast=(1_A\ot 1'_A\# a)m$
for all $m\in M, a\in A$ and give a quasi-Poisson $A$-module structure on $M$, where $\1$ is the multiplicative
identity element of $\U(A)$.
The correspondence of Poisson modules
and $\P$-modules is given similarly.
For simplicity, we write $a\ot b' \#\al+J$ as
$a\ot b' \#\al$ when no confusion can arise.

In \cite{BY}, the authors consider the following bicomplex by taking tensor product of the
bar resolution of the $A^e$-module $A$ and the projective resolution
of the trivial $\U(A)$-module $\K$
{\footnotesize
\begin{align}\label{bicomplex of quasi-Poisson modules}
\begin{CD}
&& \cdots && \cdots &&\cdots\\
&&@VVV @VVV @VVV\\
0 @<<< A^4\ot \U(A) @<{\eta^H_{2,0}}<< A^4\ot \U(A)\ot \wedge^1 @<{\eta^H_{2,1}}<<  A^4\ot \U(A)\ot \wedge^2 @<<< \cdots\\
&& @V{\eta^V_{1,0}}VV @V{\eta^V_{1,1}}VV @V{\eta^V_{1,2}}VV\\
0 @<<< A^3\ot \U(A) @<{\eta^H_{1,0}}<< A^3\ot \U(A)\ot \wedge^1 @<{\eta^H_{1,1}}<<  A^3\ot \U(A)\ot \wedge^2 @<<< \cdots\\
&& @V{\eta^V_{0,0}}VV @V{\eta^V_{0,1}}VV @V{\eta^V_{0,2}}VV\\
0 @<<< A^2\ot \U(A) @<{\eta^H_{0,0}}<< A^2\ot \U(A)\ot \wedge^1 @<{\eta^H_{0,1}}<<  A^2\ot \U(A)\ot \wedge^2 @<<< \cdots\\
&& @VVV @VVV @VVV\\
&& 0 && 0 &&0
\end{CD}
\end{align}}
\noindent Its total complex gives a free resolution of $A$ as a $\Q$-module. For any quasi-Poisson module $M$, after applying the functor $\Hom_\Q(-,M)$
to the total complex, we obtain the
quasi-Poisson complex with coefficients in $M$, whose cohomology group is isomorphic the extension group $\Ext_\Q^\ast(A,M)$, and called the \emph{quasi-Poisson cohomology group with coefficients in} $M$, see \cite[Section 3]{BY}.

\section{FGV-Poisson cohomology}

\subsection{A complex of Poisson modules}

Let $(A,\cdot,\{-,-\})$ be a Poisson algebra and $\P$ the Poisson
enveloping algebra of $A$. Set $$C_{i, j}=\begin{cases}
\P\ot \wedge^{j}, &i=0, j\ge 0,\\
\P\ot A^{i+1}\ot \wedge^{j-1}, &i\ge 1, j\ge 1,\\
0, &{\rm otherwise}.
\end{cases}$$
Clearly, $C_{i, j}$ is a left $\P$-module for any $i, j\ge 0$. We consider the following diagram
{\footnotesize $$\begin{CD}
&& \cdots&& \cdots&& \cdots\\
&& @VVV @VVV @VVV\\
 0 @<<< \P\ot A^3 @<{\delta^H_{2,1}}<< \P\ot A^3\ot \wedge^1 @<{\delta^H_{2,2}}<< \P\ot A^3\ot \wedge^2 @<<< \cdots\\
&& @V{\delta^V_{1,1}}VV @V{\delta^V_{1,2}}VV @VV{\delta^V_{1,3}}V\\
0 @<<< \P\ot A^2 @<{\delta^H_{1,1}}<< \P\ot A^2\ot \wedge^1 @<{\delta^H_{1,2}}<< \P\ot A^2\ot \wedge^2 @<<< \cdots\\
&& @V{\delta^V_{0,1}}VV @V{\delta^V_{0,2}}VV @VV{\delta^V_{0,3}}V\\
\P @<{\delta^H_{0,0}}<< \P\ot \wedge^1 @<{\delta^H_{0,1}}<< \P\ot  \wedge^2 @<{\delta^H_{0,2}}<< \P\ot \wedge^3 @<<< \cdots\\
&& @VVV @VVV @VVV\\
&& 0 && 0 && 0
\end{CD}$$}
where $\delta^V_{i,j}: C_{i+1,j}\to C_{i,j}$ and $\delta^H_{i,j}: C_{i,j+1}\to C_{i,j}$
are $\P$-homomorphisms given by
\begin{align*}
& \delta^V_{i, j}((1_A\ot 1'_A\# \1)\ot a_0\ot \cdots\ot a_{i}\ot \omega^j)\\
=& (-1)^i((a_0\ot 1'_A\# \1)\ot a_1\ot \cdots \ot a_i\ot \omega^j\\
&+\sum_{k=0}^{i-1}(-1)^{k+1} (1_A\ot 1'_A\# \1)\ot a_0\ot \cdots \ot a_ka_{k+1}\ot \cdots \ot a_{i}\ot \omega^j\\
& +(-1)^{i+1}(1\ot a'_i\#\1)\ot a_0\ot \cdots\ot a_{i-1}\ot \omega^j)
\end{align*}
for $i\ge 2$ and $j\ge 0$,
\begin{align*}
&\delta^V_{1, j}((1_A\ot 1'_A\# \1)\ot a_0\ot a_1\ot \omega^j)\\
=& (a_0\ot 1'_A\# \1)\ot (a_1\wedge \omega^j)
-(1_A\ot 1'_A\# \1)\ot (a_0a_1\wedge \omega^j)
+(1_A\ot a'_1\#\1)\ot (a_0\wedge \omega^j)
\end{align*}
for $i=0$, $j\ge 1$ and $\omega^j\in \wedge^j$, and
\begin{align*}
&\delta^{H}_{i,j}((1_A\ot 1'_A\#\1)\ot \theta^i\ot (x_0\wedge \cdots \wedge x_j))\\
=& \sum_{k=0}^j(-1)^{k}(1_A\ot 1'_A\# x_k)\ot \theta^i\ot
(x_0\wedge\cdots \widehat{x_k} \cdots \wedge x_j)\\
&+\sum_{0\le p<q\le j} (1_A\ot 1'_A\# \1)\ot \theta^i\ot
(\{x_p,x_q\}\wedge x_0\wedge\cdots \widehat{x_p}\cdots \widehat{x_q} \cdots \wedge x_j)
\end{align*}
for any $i\ge 0$, $j\ge 1$ and $\theta^i\in A^i$.

\begin{prop}\label{cpx chi} The diagram
$C_{\bullet, \bullet}=(C_{i,j},\delta^H_{i,j},\delta^{V}_{i,j})$
is a bicomplex of free $\P$-modules, whose total complex is denoted by
$\chi_\bullet(A)$.
\end{prop}

\begin{proof}
Observe that $\P$ is a quotient algebra of $\Q$ and hence can be viewed
as a $\Q$-module.
By applying the functor $\P\uoQ -$
the bicomplex \eqref{bicomplex of quasi-Poisson modules}
and the natural isomorphism of left $\P$-modules
$\P\uoQ \Q \ot (A^i\ot \wedge^j)
\cong \P\ot A^i\ot \wedge^j$,
we obtain the following bicomplex of Poisson modules
{\footnotesize
\begin{align}\label{bicomplex of Poisson modules}
\xymatrix{
 & \cdots\ar[d] & \cdots \ar[d] & \cdots \ar[d] & \\
0 & \P\ot A^3\ar[l]\ar[d]_{\bar\eta^V_{2,0}}
& \P\ot A^3\ot \wedge^1 \ar[l]_{\bar\eta^H_{3,0}}\ar[d]_{\bar\eta^V_{2,1}}
& \P\ot A^3\ot \wedge^2 \ar[l]_{\bar\eta^H_{3,1}}\ar[d]_{\bar\eta^V_{2,2}} & \cdots\ar[l]\\
0 & \P\ot A^2 \ar[l]\ar@{-->}[d]_{\bar\eta^V_{1,0}}\ar[ddr]|(.3){\delta^V_{0,1}}
& \P\ot A^2\ot \wedge^1 \ar[l]_{\bar\eta^H_{2,0}}\ar@{-->}[d]_{\bar\eta^V_{1,1}}\ar[ddr]|(.3){\delta^V_{0,2}}
& \P\ot A^2\ot \wedge^2 \ar[l]_{\bar\eta^H_{2,1}}\ar@{-->}[d]_{\bar\eta^V_{1,2}}\ar[ddr]|(.3){\delta^V_{0,3}}
& \cdots\ar[l]\\
0 & \P\ot A \ar@{-->}[l]\ar@{-->}[d]\ar@{-->}[dr]_{\varphi_0}
& \P\ot A\ot \wedge^1 \ar@{-->}[l]\ar@{-->}[d]\ar@{-->}[dr]_{ \varphi_1}
 & \P\ot A\ot \wedge^2 \ar@{-->}[l]\ar@{-->}[d]\ar@{-->}[dr]_{\varphi_2}
 & \cdots\ar@{-->}[l]\\
0 & \P\ar[l]\ar[d]
& \P\ot \wedge^1 \ar[l]^{\bar\eta^H_{0,0}}\ar[d]
& \P \ot \wedge^2 \ar[l]^{\bar\eta^H_{0,1}}\ar[d]
& \cdots,\ar[l]\\
& 0 & 0  & 0 &
}\end{align}}
where $\bar\eta^V_{i,j}$, $\bar\eta^H_{i,j}$ are induced by $\id_\P\ot \eta^V_{i,j}$ and $\id_\P\ot \eta^H_{i,j}$,
respectively.
Erasing all of dashed arrows, we obtain the diagram
$C_{\bullet,\bullet}$, and
$\delta_{0,j}^V=\varphi_{j-1}\bar\eta_{1,j-1}^V$,
$\delta_{i,j}^V=\bar\eta_{i+1,j-1}^V$ ($i\ge 2$),
and
$\delta_{0,j}^H=\bar\eta_{0,j}$,
$\delta_{i,j}^H=\bar\eta_{i+1,j-1}^H$ ($i\ge 1$),
where
\[\varphi_j( (1_A\ot 1'_A\# \1) \ot a \ot \omega^j)
=(1_A\ot 1'_A\# \1) \ot (a \wedge \omega^j).\]
It is easily seen that
$\delta^H_{i,j}\delta_{i,j+1}^H=0$, $\delta_{i,j}^V\delta_{i+1,j}^V=0$,
$\delta_{i,j-1}^V\delta_{i,j}^H+\delta_{i-1,j}^H\delta_{i,j}^V=0$ for all $i\ge 2$,
and
$\delta_{0,j}^V\delta_{1,j}^V=\varphi_{j-1}\bar\eta_{1,j-1}^V\bar\eta_{2,j-1}^V=0$.
It remains to check $\delta_{0,1}^V\delta_{1,j}^H+\delta_{0,j}^H\delta_{0,j+1}^V=0$ for all $j\ge 1$.
We only show the case $j=1$, and the general cases can be proved similarly. By definition we have
\begin{align*}
&(\delta_{0,1}^V\delta_{1,1}^H)((1_A\ot 1'_A\# \1) \ot a\ot b\ot x)\\
=& \delta_{0,1}^V((1_A\ot 1'_A\# x) \ot a\ot b)\\
=&\delta_{0,1}^V[(1_A\ot 1'_A\# x)(1_A\ot 1'_A\# \1\ot a\ot b)\\
&-1_A\ot 1'_A\# \1\ot \{x,a\}\ot b-1_A\ot 1'_A\# \1\ot a\ot \{x,b\}]\\
=& (1_A\ot 1'_A\# x)[a\ot 1_A\#\1\ot b-1_A\ot 1'_A\#\1 \ot ab+1_A\ot b'\#\1\ot a]\\
&-\{x,a\}\ot 1'_A\# \1\ot b+1_A\ot 1'_A\# \1\ot \{x,a\}b-1_A\ot b'\#\1\ot\{x,a\}\\
&-a\ot 1'_A\# \1\ot \{x,b\}+1_A\ot 1'_A\# \1\ot a\{x,b\}-1_A\ot \{x,b\}'\#\1\ot a\\
=& (a\ot 1'_A\# x)\ot b-(1_A\ot 1'_A\# x) \ot ab+(1_A\ot b'\# x)\ot a
+(1_A\ot 1'_A\# \1)\ot \{x,a\}b\\
&-(1_A\ot b'\#\1)\ot\{x,a\}+(1_A\ot 1'_A\# \1)\ot a\{x,b\}-
(1_A\ot \{x,b\}'\#\1)\ot a\\
=& -\delta_{0,1}^H((a\ot 1'_A\#\1) \ot b\wedge x-
(1_A\ot 1'_A\#\1) \ot ab\wedge x+(1_A\ot b'\#\1) \ot a\wedge x
)\\
=&- (\delta_{0,1}^H\delta_{0,2}^V)(1_A\ot 1'_A\# \1 \ot a\ot b\ot x),
\end{align*}
which completes the proof.
\end{proof}

\subsection{FGV-Poisson cohomology groups}

Let $M$ be a Possion module over $A$. Recall that the $n$-th FGV-Poisson cohomology groups of $A$ with coefficients in $M$
in \cite{FGV}, denoted by $\HP^n(A, M)$,  is defined as the $n$-th cohomology group
of the total complex  of the following bicomplex $\widetilde{C}^{\bullet, \bullet}(A; M)$
{\footnotesize
$$\begin{CD}
&& \cdots  && \cdots && \cdots\\
&& @AAA  @AAA  @AAA\\
0@>>>\Hom(A^3,M) @>{\delta_H}>> \Hom(A^3\ot \wedge^1, M) @>{\delta_H}>> \Hom(A^3\ot \wedge^2,M) @>>> \cdots\\
&&@A{\delta_V}AA @A{\delta_V}AA @AA{\delta_V}A\\
0@>>>\Hom(A^2,M) @>{\delta_H}>> \Hom(A^2\ot \wedge^1, M) @>{\delta_H}>> \Hom(A^2\ot \wedge^2,M) @>>> \cdots\\
&&@A{\delta_v}AA @A{\delta_v}AA @AA{\delta_v}A\\
M@>{\delta_h}>>\Hom(\wedge^1,M) @>{\delta_h}>> \Hom(\wedge^2, M) @>{\delta_h}>> \Hom(\wedge^3,M) @>>> \cdots\\
&&@AAA @AAA @AAA\\
&& 0  && 0 && 0
\end{CD},$$}
where $\delta_h$ is the Chevalley-Eilenberg differential, $\delta_V$ is
the Hochschild differential, and $\delta_v$ is the composition
of the natural homomorphism $\Hom(\wedge^j, M) \hookrightarrow \Hom(A\ot \wedge^{j-1},M)$ and the Hochschild differential $\Hom(A\ot \wedge^{j-1},M)\to \Hom(A^2\ot\wedge^{j-1},M)$, and $\delta_H: \Hom(A^i\ot \wedge^{j-1},M) \to \Hom(A^i\ot \wedge^{j},M)$ is given by
\begin{align*}
\begin{split}
&(\delta_H f)(a_1\ot \cdots \ot a_i \ot (x_1\wedge \cdots \wedge x_j))\\
=& \sum_{l=1}^j(-1)^{l+1}
\bigg(\{x_l, f(a_1\ot \cdots \ot a_i\ot(x_1\wedge \cdots  \hat{x}_l\cdots \wedge x_j)\}_\ast\\
&~~~~~~-\sum_{t=1}^i f
(a_1\ot \cdots \ot \{x_l,a_t\}\ot\cdots\ot a_i \ot
(x_1\wedge \cdots  \hat{x}_l\cdots \wedge x_j))\bigg)\\
&+\sum_{1\le p<q\le j}(-1)^{p+q}f
(a_1\ot \cdots \ot a_i \ot (\{x_p,x_q\}\wedge
x_1\wedge\cdots\hat{x}_p\cdots \hat{x}_q \cdots \wedge x_j)).
\end{split}
\end{align*}

\begin{rmk}
Note that the differential $\delta_H$ is essentially  a Lie algebra differential. In fact, $\Hom(A^i\ot\wedge^j, M) \cong \Hom(\wedge^j, \Hom(A^i, M))$ and $\Hom(A^i, M)$ is a Lie module over $A$ with the action
given by
\begin{align*}\{x, f\}_\ast(a_1\ot \cdots \ot a_i)=&
\{x, f(a_1\ot \cdots \ot a_i)\}_\ast-
\sum_{t= 1}^i f(a_1\ot \cdots \{x,a_t\}\cdots \ot a_i)
\end{align*}
for any $f\in \Hom(A^i, M)$ and any $a_1\ot \cdots \ot a_i\in A^i$.
\end{rmk}

The total complex of $\widetilde{C}^{\bullet, \bullet}(A; M)$ is of form
\begin{align*}
 & 0\to M\xr{d^0} \Hom(A,M)
\xr{d^1}\Hom(A^2\oplus \wedge^2,M)
\xr{d^2}\Hom(\underset{{i+j=3}\atop{i\neq 1}}{\oplus}A^{i}\ot \wedge^j,A)\\
&\to \cdots \to \Hom(\underset{{i+j=n}\atop{i\neq 1}}{\oplus}A^{i}\ot \wedge^j, M)
\xr{d^n} \Hom(\underset{{i+j=n+1}\atop{i\neq 1}}{\oplus}A^{i}\ot \wedge^j, M)
\to \cdots,
\end{align*} and we call it the \emph{Poisson complex with coefficients in $M$}.
The Poisson complex with coefficients in $A$ is called the \emph{Poisson complex of $A$},
and we simply denote $\HP^n(A;A)$ by $\HP^n(A)$ and call it the $n$-th \emph{FGV-Poisson
cohomology group} of $A$.

\begin{rmk} The lower dimensional Poisson cohomology groups have the following explicit interpretation.

Clearly, $\HP^0(A;M)=\{m\in M~|~\{a, m\}_\ast=0, \forall a\in A \}$.
In particular, $\HP^0(A)$ is the center of the Lie algebra $A$.

$\HP^1(A;M)=\PDer(A;M)/\IPDer(A;M)$ is the outer Poisson
derivations of $A$ with coefficients in $M$, where $\PDer(A;M)\}$ is the set of Poisson derivations
and $\IPDer(A;M)=\{\{-,m\}_*\mid m\in M\}$ the set of inner ones. Recall that a linear map from $A$ to $M$
is called a \emph{Poisson derivation} if it is simultaneously
a derivation in the Lie algebra sense and in the associative sense.
\end{rmk}

For any $f=(f_1,f_0)\in \Ker d^2$, we may define a new Poisson algebra
$A\ft M$, which is called the \emph{extension of $A$ by $M$ along $f$}.
As a $\K$-vector space, $A\ft M=A\oplus M$; and the associative
multiplication and the Lie bracket are given by
\begin{align*}
(a,x)\cdot (a', x')&=(aa', ax'+xa'+f_1(a\ot a')),\\
\{(a,x),(a',x')\}&=(\{a,a'\},\{a,x'\}_\ast-\{a',x\}_\ast+f_0(a\wedge a')\}.
\end{align*}
We have the following standard result, compare also with the deformation theory.

\begin{prop} For any $f\in \Ker d^2$,  $A\ft M$ is a Poisson algebra.
Moreover, for any $f,g\in \Ker d^2$ with $\overline{f}
=\overline{g}$ in $\HP^2(A,M)$, there is an isomorphism of Poisson algebras $A\ft M\simeq A\gt M$.
\end{prop}

Applying the complex $\chi_\bullet(A)$, we can interpret the FGV-Poisson cohomology of $A$ by the derived functor  $\mathrm{RHom}_{\P}(\chi_\bullet(A),-)$.

\begin{thm}\label{PC by der func}
 Let $A$ be a Poisson algebra and $M$ a Poisson module over $A$. Then
\[\HP^n(A, M)\cong \H^n \Hom_\P(\chi_\bullet(A), M).\]
\end{thm}
\begin{proof}
By identifying $\Hom_\P(\P\ot A^i\ot \wedge^j, M)$ with $\Hom(A^i\ot \wedge^j,M)$, we easily deduce
that the double complex $PC^{\bullet,\bullet}(A; M)$ is isomorphic to $\widetilde{C}^{\bullet, \bullet}(A; M)$.
The rest of the proof is obvious.
\end{proof}

The theorem says that the FGV-Poisson cohomology is essentially defined by derived functors. An advantage is that the well-developed results and tools in the usual cohomology theory for algebras apply for Poisson cohomology theory as well. For instance, we may obtain the following long exact sequence of Poisson cohomology groups from a short exact sequence of Poisson modules, which is not so obvious from the original definition given by Flato, Gerstenhaber and Voronov.

\begin{prop}
Let $A$ be a Poisson algebra, and
$0\to M_1 \to M_2\to M_3\to 0$
an exact sequence of Poisson $A$-modules. Then we have a long exact sequence
\begin{align*}
0& \to \HP^0(A, M_1) \to \HP^0(A, M_2) \to \HP^0(A, M_3)\to \HP^1(A, M_1) \to \cdots\\
& \to \HP^n(A, M_1) \to \HP^n(A, M_2) \to \HP^n(A, M_3)\to \HP^{n+1}(A, M_1)\to \cdots
\end{align*}
of FGV-Poisson cohomology groups.
\end{prop}
\begin{proof} Let $\P$ be the Poisson enveloping algebra of $A$. Then each Poisson module can be viewed as a $\P$-module, and the sequence $0\to M_1 \to M_2\to M_3\to 0$
gives a distinguished triangle in the derived category $D(\P)$, where each $M_i$ is regarded as a stalk complex of concentrated degree $0$. Applying the cohomology functor $\Hom_{D(\P)}(\chi_\bullet(A), -)$, we obtain
the desired long exact sequence.
\end{proof}

\section{Formal Deformations and deformation quantization}

In \cite[Section 6]{FGV}, a deformation theory for Poisson algebras has been developed.
It is worth pointing out that, even for commutative Poisson algebras, the
noncommutative version of deformation theory is quite useful. In fact, Kontsevich's
deformation quantization is explained as a special case of formal deformations in
the noncommutative sense, and by using Poisson cohomology group we can give some necessary condition for the existence of deformation quantization.

\subsection{Formal Poisson deformations}
Let $(A,\cdot,\{-,-\})$ be a Poisson algebra over $\K$.
Let $\K[[t]]$ and $A[[t]]$ be the formal power series ring
in one variable $t$ with coefficients in $\K$ and $A$, respectively.
A \emph{formal deformation} of the Poisson algebra $A$ means a $\K[[t]]$-Poisson algebra
$(A[[t]], m_t, l_t)$ such that $(A[[t]], m_t)$ and $(A[[t]], l_t)$ are the formal deformations of the associative algebra $(A, \cdot)$
and the Lie algebra $(A, \{-, -\})$, respectively. Clearly the associative multiplication $m_t$ and the Lie bracket $l_t$ are determined by their restrictions to the subset $A$. We may write
\begin{align*}
m_t(a, b) & = m_0(a, b)+tm_1(a, b)+t^2m_2(a, b)+\cdots,\\
l_t(a, b) & = l_0(a, b)+tl_1(a, b)+t^2l_2(a, b)+\cdots,
\end{align*}
for any $a, b\in A$, where $m_i, l_i\colon A\times A\to A$ are $\K$-bilinear maps. By definition, $m_0(a, b)=ab$ and $l_0(a, b)=\{a, b\}$.

As we expected, the FGV-Poisson cohomology controls the deformation of Poisson algebra,
which essentially goes to Flato, Gerstenhaber and
Voronov in \cite[Section 6]{FGV}.

 \subsection{Deformation quantization} In this subsection, we always assume that $\K$ is an algebraically closed field
 and $A$ is a commutative associative algebra.
We will explain in the sequel how to understand Kontsevich's deformation quantization as a special case of formal deformation of
Poisson algebras in our sense. We refer to \cite{DMZ,Ko,St} for more details about deformation quantization theory.

Let $(A, \cdot)$ be a commutative associative $\K$-algebra, and $(A[[t]], m_t)$ a formal deformation of $A$ as an associative algebra.
As usual, $m_t$ is given by a family of $\K$-bilinear maps $\{m_i: A\times A\to A\mid i=0, 1, \cdots, \}$ with \[m_t(a,b)=m_0(a,b)+tm_1(a,b)+t^2m_2(a,b)+\cdots\]
for any $a,b\in A$, and $m_0(a,b) = ab$.
Consider the bilinear map $\{-, -\}\colon A\times A\to A$ given by $\{a,b\}= m_1(a,b)-m_1(b,a)$.
It is easy to check that $P = (A, \cdot, \{-,-\})$ is a commutative Poisson algebra.                         ¡¡¡¡¡¡¡¡¡¡¡¡¡¡

\begin{Def}\cite[Definition 8.4]{DMZ}
Let $A, P$ be as above. Then $P$ is called the \emph{classical limit} of the associative product $m_t$, and $(A[[t]], m_t)$
 is called a \emph{deformation quantization} of the Poisson algebra $P$.
\end{Def}

Recall that a Poisson algebra $(B, \cdot, \{-,-\})$ over a commutative ring $R$ is said to be \emph{standard} if $\lambda\{-,-\}=\mu[-, -]$ for some $\lambda, \mu\in R$, where
$[a, b]=ab-ba$ is the commutator and $\lambda$ is not a zero divisor in $R$. In this case, we also write $\{-,-\}=\frac {\mu} {\lambda}[-, -]$.

Suppose that $(A[[t]], m_t)$ is a deformation quantization of the Poisson algebra $(A, \cdot, \{-, -\})$. We consider the standard Poisson algebra $A[[t]]$ with the bracket given by \[l_t = \frac{1}{t}[-, -]_t\colon A[[t]]\times A[[t]]\to A[[t]], \]
where $[-, -]_t$ is the commutator of $m_t$. Clearly, the Lie bracket $l_t$ is given by
\[l_t(a, b)=\sum\limits_{i\ge 0}(m_{i+1}(a, b)-m_{i+1}(b, a))t^{i}\]
for any $a, b\in A$, since $A$ is a commutative algebra and hence $m_0(a,b)-m_0(b,a)=0$. It follows that
$(A[[t]], m_t, l_t)$ is a formal deformation of the Poisson algebra $(A, \cdot, \{-, -\})$. Thus \textbf{a deformation quantization of a commutative Poisson algebra is a formal Poisson deformation}.

More generally, we may introduce the notion of $n$-deformation quantization.
\begin{prop}
Let $A$ be a commutative algebra over $\K$ and $(A[[t]], m_t)$ a formal deformation of $A$.
If $m_i(a, b)=m_i(b, a)$ for any $a, b\in A$ and any $0\le i\le n-1$,
then $\{a, b\}=m_n(a, b)-m_n(b, a)$ gives a Poisson structure on $A$.
\end{prop}
\begin{proof} By assumption, $[A,A]_t\subseteq t^nA[[t]]$, thus we may consider a standard Poisson structure $l_t=\frac{1}{t^n}[-, -]_t$ on $A[[t]]$, where $[-, -]_t$
is the commutator of $m_t$.
The bracket $l_t$ can be expressed as
$l_t(a, b)=\sum\limits_{i\ge 0}l_i(a, b)t^i$
with each
\[l_i(a, b)=m_{n+i}(a, b)-m_{n+i}(b, a)\]
for any $a, b\in A$.
By comparing the constant terms of two sides of the following equations
\begin{align*}
&l_t(a, b)+l_t(b, a)=0,\\
&l_t(l_t(a, b), c)+l_t(l_t(b, c), a)+l_t(l_t(c, a), b)=0,\\
&l_t(m_t(a, b), c)=m_t(a, l_t(b, c))+m_t(l_t(a, b), c),
\end{align*}
we get
\begin{align*}
&l_0(a, b)+l_0(b, a)=0,\\
&l_0(l_0(a, b), c)+l_0(l_0(b, c), a)+l_0(l_0(c, a), b)=0,\\
&l_0(ab, c)=al_0(b, c)+l_0(a, b)c,
\end{align*}
and hence $l_0$ gives a Poisson structure on $A$.
\end{proof}
In this situation, we denote $\{-, -\}=l_0$ and call the formal deformation $(A[[t]], m_t)$ an \emph{$n$-deformation quantization} of the Poisson algebra $(A, \cdot, \{-, -\})$.

\begin{thm}\label{dq and fd} Let $P=(A, \cdot, \{-,-\})$ be a commutative Poisson algebra.

$(1)$ If $P$ admits an $n$-deformation quantization for some $n\ge 1$, then $P$ can be deformed to a standard Poisson algebra.

$(2)$ Assume further that each formal deformation $(A[[t]], m_t)$
of $(A, \cdot)$ has only standard Poisson structures.
Then $P$ has an $n$-deformation quantization for some positive integer $n$ if and only if $P$ has a formal deformation.
\end{thm}

\begin{proof} (1) Suppose that the Poisson algebra $(A, \cdot, \{-,-\})$ has an $n$-deformation quantization $(A[[t]], m_t)$
satisfying $m_i(a, b)=m_i(b, a)$ for any $a, b\in A$ and $i=0, 1, \cdots, n-1$. By definition, we have
$\{a, b\}=m_n(a, b)-m_n(b, a)$ for all $a, b\in A$.
We consider the standard Poisson algebra $(A[[t]], m_t, l_t)$ with
\[l_t(a, b)=\dfrac{1}{t^n} (m_t(a, b)-m_t(b, a)).\]
Since $ab=m_0(a, b)$ and $\{a, b\}=m_n(a, b)-m_n(b, a)=l_0(a, b)$,
the standard Poisson algebra $(A[[t]], m_t, l_t)$ is just the formal deformation
of $(A, \cdot, \{-, -\})$.

(2)  By (1), we only need to prove that the part of ``if ".
Suppose that $(A, \cdot, \{-,-\})$ has a formal deformation $(A[[t]], m_t, l_t)$, where
$m_t(a, b)=\sum\limits_{i=0}^\infty m_i(a, b)t^i$ and $l_t(a, b)=\sum\limits_{i=0}^\infty l_i(a, b)t^i$.
Since $(A[[t]], m_t)$ has only standard Poisson structures, there exists some $\la(t), \mu(t)\in \K[[t]]$ such that $l_t(a, b)=\frac{\mu(t)}{\la(t)}(m_t(a, b)-m_t(b, a))$.
Without loss of generality, we assume that $\lambda(t)=t^n$ for some $n\ge 1$, and $$\mu(t) =\mu_0+ \mu_1 t+\cdots+\mu_i t^i+\cdots\in \K[[t]]$$ with $\mu_i\in \K, i\ge0$ and $\mu_0\neq 0$.

By comparing the both sides of the above expression of $l_t(a, b)$, we obtain that $m_i(a, b)=m_i(b, a)$ for $i=0, 1, \cdots, n-1$ and $$\{a, b\}=l_0(a, b)=\mu_0(m_n(a, b)-m_n(b, a)).$$ Observe that for any $\nu\in \K$,
the family of $\K$-bilinear maps $m'_i(a, b)=\nu^im_i(a, b)$ ($i\ge 0$) forms a formal deformation of $A$.
So we can choose $\nu=\mu_0^{\frac{1}{n}}$ since $\K$ is algebraically closed. It follows that
$\{a, b\}=m'_n(a, b)-m'_n(b, a)$ for any $a, b\in A$ and therefore $A$ has an $n$-deformation quantization.
\end{proof}


\begin{ex}\label{ex-intdom} Let $A$ be a Poisson algebra which is an integral domain as an algebra, then $A$ has an $n$-deformation quantization for some positive integer $n$ if and only if $A$ has a formal deformation. In fact, any formal deformation $(A[[t]], m_t)$ of an integral domain $A$ is prime, and hence has only standard Poisson structures by a result of Farkas and Letzter, see \cite[Theorem 1.2]{FL} for more details. Then Theorem \ref{dq and fd} does work.
\end{ex}

By Theorem \ref{dq and fd} (1), we can give a necessary condition of
the existence of deformation quantization of a Poisson algebra.

\begin{prop}\label{condition for dq and fd}
Let $P=(A, \cdot, \{-,-\})$ be a nontrivial commutative Poisson algebra. If $\HP^2(A)=0$, then $P$ has no
deformation quantization.
\end{prop}
\begin{proof}
Assume that $(A[[t]], m_t)$ is a deformation quantization of $P$. Then
$(A[[t]], m_t, l_t)$ is a formal deformation of $(A, \cdot, \{-,-\})$, where
\[l_t(a,b)=\frac{1}{t}(m_t(a,b)-m_t(b,a))\]
for all $a, b\in A$.
Since $\HP^2(A)=0$, the formal deformation $(A[[t]], m_t, l_t)$ is equivalent to the trivial one, that
is, there exists an isomorphism
\[g\colon (A[[t]], m_t, l_t)\to (A[[t]], m'_t, l'_t)\] of Poisson $\K[[t]]$-algebras
such that $g(a)\in a+tA[[t]]$ for any $a\in A$, where $m'_t(a, b) =ab$, $l'_t(a,b) =\{a,b\}$ for $a, b\in A$.
Therefore,
\begin{align*}
\{a,b\}= & l'_t(a,b)=g(l_t(g^{-1}(a), g^{-1}(b)))\\
=& g\frac{1}{t}(m_t(g^{-1}(a), g^{-1}(b))-m_t(g^{-1}(b), g^{-1}(a)))\\
=& \frac{1}{t}(m'_t(a,b)-m'_t(b,a))=\frac{1}{t}(ab-ba)=0,
\end{align*}
which leads to a contradiction.
\end{proof}

\section{Comparing FGV-Poisson and LP cohomologies}
Recall the definition of Lichnerowicz-Poisson cohomology for a commutative Poisson
algebra, see \cite{Gu,Hu,Li, Pic, St} for detail.
We set $\xi^i(A)=0$ for any $i<0$,
$\xi^0(A)=A$ and $\xi^1(A)=\Der(A)$. For $i\ge 2$, let
$\xi^\bullet(A)$ be the subspace of $\Hom(\wedge^\bullet, A)$ consisting of all skew-symmetric multiderivations of $A$.
The subcomplex $(\xi^\bullet(A), d)$ of the Chevalley-Eilenberg complex $(\Hom(\wedge^\bullet, A), d)$
is called the \emph{Lichnerowicz-Poisson complex of $A$},
or simply \emph{LP-complex},
and its $n$-th cohomology group, denoted by $\H_{LP}^n(A)$, is called the \emph{$n$-th Lichnerowicz-Poisson (LP for short) cohomology group of $A$},  see \cite{Pic}.

\begin{rmk} Let $A$ be a Poisson algebra, then by definition we have
\[ \HP^0(A)= \H_{LP}^0(A),\ \ \HP^1(A)= \H_{LP}^1(A).\]
In fact, easy calculation shows that the $0$-th LP-cohomology $\H_{LP}^0(A)$ is the center of the Lie algebra $A$,  and $\H_{LP}^1(A)$ is exactly the set of Poisson derivations.
\end{rmk}

Moreover, by direct calculation, we show that the LP-cohomology
relates to the FGV-Poisson cohomology closely.

We simply denote the Poisson bicomplex $PC^{\bullet,\bullet}(A, A)$ by $PC^{\bullet, \bullet}$.
Considering the spectral sequence induced by the first filtration of the bicomplex $PC^{\bullet, \bullet}$, we obtain
a complex
$$0\to A\to \H_\I^0(PC^{\bullet,1}) \to \H_\I^0(PC^{\bullet,2})\to  \cdots\to \H_\I^0(PC^{\bullet,j}) \to \H_\I^0(PC^{\bullet,j+1})\to \cdots,$$
where
\[\H^0_\I(PC^{\bullet,j})=\{f\in \Hom(\wedge^j,A)\mid \delta_v(f)=0\}\]
is the space of all skew-symmetric $n$-fold derivations of $A$.

\begin{prop} Keep the above notation, we have
\[ \H_{LP}^n(A)\cong E_{\II 2}^{0, n},\]
where $E_{\II 2}^{0, n}$ is the $(n, 0)$-term in the spectral sequence given by the second filtration of the FGV-Poisson bicomplex.
\end{prop}

\begin{rmk}
We would like to mention that the FGV-Poisson cohomology controls the formal deformations of a Poisson algebra (see Section 4.1),
while the LP-cohomology controls those deformations which only deform the Lie bracket.
\end{rmk}

To study the further relation between FGV-Poisson cohomology and LP-cohomology,
we consider the first filtration of the FGV-Poisson bicomplex,
whose $(n, 0)$-term is given by
\begin{align*}
\zeta^1(A)=& \{f\in \Hom_{\Bbbk}(A, A)\mid \{x, f(y)\}+\{y, f(x)\}-f(\{x, y\})=0\},\\
\zeta^n(A)=& \{f\in \Hom_{\Bbbk}(A^n, A)\mid \{x, f(a_1, \cdots, a_n)\}=\sum\limits_{i=1}^n f(a_1, \cdots, \{x, a_i\}, \cdots, a_n) \},
\end{align*}
for $n\ge 2$.

The complex $\zeta^\bullet$ is a subcomplex of the Hochschild complex. We call $\zeta^\bullet$ the  \emph{Hochschild-Poisson complex}, and its $n$-th cohomology $\H_{HP}^n(A)$ the $n$-\emph{Hochschild-Poisson cohomology group}.

\begin{prop}\label{FGV=HP+LP}
Let $A$ be a commutative Poisson algebra over $\Bbbk$, which is finitely generated and smooth as an associative algebra. Then there is an isomorphism
\[\HP^2(A)\cong \H_{HP}^2(A)\oplus \H_{LP}^2(A).\]
\end{prop}
\begin{proof}
Let $\al\in \Hom_{\Bbbk}(A^2, A)$ and $\beta\in \Hom_{\Bbbk}(\wedge^2, A)$ such that $(\al, \beta)$ is a FGV-Poisson 2-cocycle.
By definition, $\al$ is a Hochschild 2-cocycle.
Since $A$ is finitely generated and smooth, by HKR theorem we have $\HH^2(A)=\wedge^2_A \Der(A)$, see \cite{HKR}.
Without loss of generality, we may assume that
\[\al(f, g)=a(\delta_1(f)\delta_2(g)-\delta_1(g)\delta_2(f))\]
for some $a\in A$ and $\delta_1, \delta_2\in \Der(A)$.
By the definition of FGV-Poisson cohomology, we have \begin{align*}
f\beta(g, h)-\beta(fg, h)+\beta(f, h)g
=\al(\{h, f\}, g)+\al(f, \{h, g\})-\{h, \al(f, g)\}
\end{align*}
for any $f, g, h\in A$.
Note that for $f, g$, the left hand side is symmetric and the right one is anti-symmetric. It forces that both sides equal 0, which means  that $\al$ is a Hochschild-Poisson 2-cocycle and
$\beta$ is a Lichnerowicz-Poisson 2-cocycle.

On the other hand, if $(\al, \beta)$ is a FGV-Poisson 2-coboundary, then $\al$ is a Hochschild-Poisson 2-coboundary and
$\beta$ is a Lichnerowicz-Poisson 2-coboundary. Then we obtain a well-defined linear map
\[\Psi\colon \HP^2(A)\to \H_{HP}^2(A)\oplus \H_{LP}^2(A),\ \overline{(\al,\beta)}\mapsto(\overline{\al}, \overline{\beta}).\]

Conversely, let $\al$ be a Hochschild-Poisson 2-cocycle and $\beta$ be a Lichnerowicz-Poisson 2-cocycle.
Clearly, $(\al, \beta)$ is a FGV-Poisson 2-cocycle.
If moreover, $\al$ is a Hochschild-Poisson 2-coboundary, and $\beta$ is a Lichnerowicz-Poisson 2-coboundary,
then there exist a Lie derivation $\varphi_1$ and a derivation $\varphi_2$
such that
\begin{align*}
\al(f, g)=& f\varphi_1(g)-\varphi_1(fg)+\varphi_1(f)g,\\
\beta(f, g)=& \{\varphi_2(f), g\}+\{f, \varphi_2(g)\}-\varphi_2(\{f, g\}).
\end{align*}
It is easily seen that $(\al, \beta)\in \Hom_{\Bbbk}(A^2\oplus \wedge^2, A)$ is a FGV-2-coboundary
given by $\varphi_1+\varphi_2$. Thus we obtain a map
\[\H_{HP}^2(A)\oplus \H_{LP}^2(A)\to \HP^2(A),\ (\overline{\al}, \overline{\beta})\mapsto\overline{(\al,\beta)}£¬\] which
is converse to $\Psi$, and the desired isomorphism follows.
\end{proof}

\begin{ex}
Let $A=\Bbbk[x, y]$ be a polynomial algebra in two variables $x, y$ with
the Poisson bracket given by $\{-, -\}=a_0\partial_x \wedge \partial_y$
for some $a_0\in A$,
where $\partial_x, \partial_y$ are the partial derivatives with respect
to the variables $x$ and $y$, respectively. This bracket was the original
Poisson bracket studied by many people including Poisson \cite{Po} when $\Bbbk=\mathbb{R}$.

First by HKR Theorem\cite{HKR}, $\HH^2(A)=A (\partial_x\wedge \partial_y)$.
Then easy calculation shows that the 2nd Hochschild-Poisson cohomology group is $\Bbbk(\overline{\partial_x\wedge \partial_y})$.
Applying the duality between LP-cohomology and LP-homology,
see Equation (7.2) in \cite{Hu2} or Remark 3.6 in \cite{LWW}, we know
the 2nd LP-cohomology group vanishes.
Therefore, by Proposition \ref{FGV=HP+LP}, we get $\HP^2(A)\cong \Bbbk$.

Next we consider the formal deformation and deformation quantization of $A$.
It is direct to show that $A$ has a deformation quantization $(A[[t]], m_t)$ with the star product $m_t$ given by
\[m_t(f, g)=\sum_{n=0}^\infty m_n(f, g)t^n=\sum_{n=0}^\infty \dfrac{t^n}{n!} \partial_x^n(f)\partial_y^n(g)\]
for all $f, g\in A\subset A[[t]]$. Furthermore, $A$ has a formal deformation $(A[[t]], m_t, l_t)$ in the noncommutative sense, where
\begin{align*}l_t(f, g)=\sum_{n=0}^\infty l_n(f, g)t^n&=\sum_{n=0}^\infty \dfrac{t^{n}}{(n+1)!} (\partial_x^{n+1}(f)\partial_y^{n+1}(g)-\partial_x^{n+1}(g)\partial_y^{n+1}(f))\\
&= \dfrac{1}{t}(m_t(f,g)-m_t(g,f))
\end{align*}
for all $f, g\in A$.
Since $A$ is  integral domain, we know that $(A[[t]], m_t)$ is a prime algebra.
By Example 4.4 and Proposition 4.3, this formal deformation exactly gives the above deformation quantization.
\end{ex}

\section{Connection with Quasi-Poisson cohomology groups}

In section 3, we have introduced the complex $\chi_\bullet(A)$ for a Poisson algebra $A$ and applied it to
study the Poisson cohomology. However, whether $\chi_\bullet(A)$ is quasi-isomorphic to a stalk complex of some Poisson module, is not
known yet. We will make some discussion on this question.

Let $\Q$ be the quasi-Poisson enveloping algebra of $(A,\cdot,\{-,-\})$
and $\Omega^2(A)$ be the second syzygy of $A$ as an $A^e$-module.
More precisely, $\Omega^2(A)$
is the quotient module of $A^e$-module $A^4$ modulo the submodule generated by $\{ a\ot b\ot c \ot 1_A-1_A\ot ab\ot c\ot 1_A+1\ot a\ot bc\ot 1_A-1\ot a\ot b\ot c \mid a,b, c\in A\}$. Note that $\Omega^2(A)$ is also a quotient module of $A^4$ as a $\Q$-module.
Similar to the construction of the free $\Q$-resolution of $A$ in \cite{BY},
we consider the bar resolution of the $A^e$-module $\Omega^2(A)$
\begin{align*}
\mathbb{S}_\bullet\colon \cdots \to A^{i+4} \to A^{i+3} \to \cdots \to A^5\to A^4 \to 0
\end{align*}
 of $\Omega^2(A)$, and the Koszul resolution of $\K$ as a trivial $\U(A)$-module
\begin{align*}
\mathbb{K}_\bullet\colon \cdots \to \U(A)\ot \wedge^j \to \U(A)\ot \wedge^{j-1} \to \cdots \to \U(A)\ot \wedge^1 \to \U(A)\to 0.
\end{align*}

By taking the total complex of the tensor product of $\mathbb{S}_\bullet$  and $\mathbb{K}_\bullet$, we obtain a free resolution of the $\Q$-module $\Omega^2(A)$
\begin{align*}
\mathbb{T}'_\bullet\colon \cdots & \to \underset{i+j=n}{\oplus}A^{i+4}\ot \U(A)\ot \wedge^j
\to \underset{i+j=n-1}{\oplus}A^{i+4}\ot \U(A)\ot \wedge^j \to \cdots\\
& \to A^5\ot \U(A)\oplus A^4\ot \U(A)\ot \wedge^1\to A^4\ot \U(A) \to 0.
\end{align*}

On the other hand, we may apply the functor $A^2\ot -$ to the resolution
$\mathbb{K}_\bullet$, and obtain a free resolution
of the $\Q$-module $A^2$
\begin{align*}
\mathbb{T}''_\bullet\colon \cdots  \to A^2\ot \U(A)\ot \wedge^n
\to A^2\ot \U(A)\ot \wedge^{n-1} \to \cdots
 \to A^2\ot\U(A)\to A^2 \to 0.
\end{align*}

By the definition of the Poisson bicomplex $C_{\bullet,\bullet}$, we immediately get a short exact sequence of complexes
\[0 \to \P\underset{\Q}{\otimes} \mathbb{T}''_\bullet
\to \chi_\bullet(A) \to \P\underset{\Q}{\otimes} \mathbb{T}'_\bullet \to 0,\]
thus we have the following long exact sequence, relating the homologies of the characteristic
complex with some torsion groups of quasi-Poisson modules.
\begin{prop}
Keep the above notations.
Then there is a long exact sequence
\begin{align*}
\cdots & \to \Tor^\Q_n(\P,A^2) \to \H_n(\chi_\bullet(A))
\to \Tor^\Q_{n-2}(\P, \Omega^2(A))\to \Tor^\Q_{n-1}(\P,A^2) \to \cdots\\
& \to \Tor^\Q_1(\P,A^2) \to \H_1(\chi_\bullet(A))
\to 0 \to \P\ot A^2 \to \H_0(\chi_\bullet(A))\to 0.
\end{align*}
\end{prop}

Now we turn to the calculation of Poisson cohomology groups.
Let $QC^\bullet(\Omega^2(A),M)$ be the complex
\begin{align*}
0 & \to \Hom(A^2, M) \to  \Hom(A^3\oplus A^2\ot \wedge^1, M) \to \cdots
\to \Hom(\underset{i+j=n-1}{\oplus} A^{i+2}\ot \wedge^j,M)\\
 & \to \Hom(\underset{i+j=n}{\oplus} A^{i+2}\ot \wedge^j,M)\to \cdots,
\end{align*}
which is obtained by applying the functor $\Hom_\Q(-, M)$ to the resolution $\mathbb{T}'_\bullet$.
We deduce the following easy result.

\begin{lem}\label{proj resol of omega^2(A)}
Let $M$ be a quasi-Poisson module over $A$.
Then
\[\Ext^n_\Q(\Omega^2(A),M)\cong \H^n(QC^\bullet(\Omega^2(A),M)), \forall n\ge 0.\]
\end{lem}

Using the projective resolution  $\mathbb{T}''_\bullet$, we have the following lemma.
\begin{lem}\label{proj resol of A^2}
Let $M$ be a quasi-Poisson module. Then
\[\Ext^n_\Q(A^2,M)\cong \HL^n(A, M),\]
where $\HL^n(A, M)$ is the $n$-th Lie algebra cohomology of the Lie algebra $(A, \{-, -\})$ with coefficients in $M$.
\end{lem}
\begin{proof} Here we use the fact that $A^2$ is a free $A^e$-module. In fact, applying the isomorphisms $\Hom_\Q(A^2\ot \U(A)\ot \wedge^n, M)\cong
\Hom(\wedge^n, M)$, we know that $\Hom_\Q(\mathbb{T}''_\bullet, M)$ is isomorphic to the Chevalley-Eilenberg complex
\[CE^\bullet: 0 \to M\to \Hom(\wedge^1, M) \to \Hom(\wedge^2, M)
\to \cdots \to \Hom(\wedge^n,M)\to \cdots,\]
and the desired isomorphism follows.
\end{proof}

Let $PC^\bullet(A,M)$ denote the Poisson complex of $A$ with coefficients in $M$.
Clearly we have a short exact sequence of complexes
\[0\to QC^\bullet(\Omega^2(A),M)[-2] \to PC^\bullet(A,M) \to CE^\bullet \to 0,\]
and combined with
Lemma \ref{proj resol of omega^2(A)} and Lemma \ref{proj resol of A^2},
we obtain the following long exact sequence.

\begin{thm}\label{l.e.s. of Poisson modules}
Let $A$ be a Poisson algebra and $M$ a Poisson module over $A$.
Then there exists a long exact sequence
\begin{align}\label{les.pcoh}
\begin{split}
0 & \to \HP^0(A, M) \to \HL^0(A, M) \to 0
\to \HP^1(A, M) \to \HL^1(A, M)\to \\
&\Hom_\Q(\Omega^2(A), M) \to \HP^2(A, M) \to \HL^2(A, M) \to \cdots
 \to \Ext_\Q^{n-2}(\Omega^2(A),M) \\&\to \HP^n(A,M) \to \HL^n(A, M)\to
\Ext_\Q^{n-1}(\Omega^2(A),M) \to \cdots.
\end{split}
\end{align}
\end{thm}

We recall a useful spectral sequence introduced in our previous work \cite{BY}.

\begin{prop}[\cite{BY}, Lemma 5.2] Let $M, N$ be modules over $\Q$.
Then we have a spectral sequence
\[\Ext_{\U(A)}^q (\K, \Ext_{A^e}^p(M,N))\Longrightarrow \Ext_{\Q}^{p+q}(M,N).\]

\end{prop}

In conclusion, the above results provide us a way to read the information of Poisson cohomology from the Lie algebra cohomology
and the quasi-Poisson cohomology.

\begin{ex}
Let $A$ be the $\K$-algebra of upper triangular $2\times 2$ matrices.
This algebra is known to be the path algebra of the quiver of $\mathbb{A}_2$ type.

Consider the standard Poisson algebra. Clearly, $A$ is a hereditary algebra as an associative algebra
and hence
$\HH^n(A)=0$ for all $n\ge 1$. By direct computation, we have
$\HL^0(A)=\K$, $\HL^1(A)=\K^2$, $\HL^2(A)=\K$, and $\HL^n(A)=0$ for all $n\ge 3$, where $\HL^i(A) = \Ext_{\U(A)}^i(\K, A)$ for each $i$.

Applying the above spectral sequence and by some direct calculations,
we know that $\Hom_\Q(\Omega^2(A),A)=\K^3$, $\Ext_\Q(\Omega^2(A),A)=\K^6$,
$\Ext_\Q^2(\Omega^2(A),A)=\K^3$ and $\Ext_\Q^n(\Omega^2(A),A)=0$ for all $n\ge 3$.
Now it follows from Theorem \ref{l.e.s. of Poisson modules} that
\[\HP^0(A)=\K,\ \HP^1(A)=0,\ \HP^2(A)=\K,\ \HP^3(A)=\K^5,\ \HP^4(A)=\K^3,\] and
$\HP^n(A)=0$ for all $n\ge 5$.
\end{ex}

\section{A further example: $\mathbb{M}_2(\K)$}

Let $A=\mathbb{M}_2(\K)$ be the standard Poisson algebra, where $\mathbb{M}_2(\K)$ is
the algebra of $2\times 2$ matrices with entries in $\K$. Clearly, as a Lie algebra,
$A=\K\cdot 1_A \oplus \mathfrak{s}\mathfrak{l}_2(\K)$, where $\K\cdot 1_A$ is an abelian Lie algebra of dimension 1 and $\mathfrak{s}\mathfrak{l}_2(\K)$
is the special linear Lie algebra with the standard basis
\[e=\begin{pmatrix}
0 & 1\\
0 & 0
\end{pmatrix},
f=\begin{pmatrix}
0 & 0\\
1 & 0
\end{pmatrix}, h=\begin{pmatrix}
1 & 0\\
0 & -1
\end{pmatrix}.\]

Consider the module $\sl_2{(\K)}\ot\sl_2{(\K)}$ over the Lie algebra $\sl_2{(\K)}$. It is easy to show the
following decomposition as Lie modules:
\[\mathfrak{s}\mathfrak{l}_2(\K)\ot \mathfrak{s}\mathfrak{l}_2(\K)=
V_5\oplus V_3\oplus V_1,\]
where $V_5, V_3$ and $V_1$ are simple modules of dimension $5$, $3$ and $1$, respectively.
More precisely, $V_5$ has a $\K$-basis \[\{e\ot e, h\ot e+e\ot h, e\ot f-h\ot h+f\ot e, f\ot f, h\ot f+f\ot h\},\]
$V_3$ has a $\K$-basis \[\{ h\ot e-e\ot h, e\ot f-f\ot e, h\ot f-f\ot h\},\] and $V_1$ has a $\K$-basis $\{2e\ot f+h\ot h+ 2f\ot e\}$, where
all basis elements are given by weight vectors.

By some direct computation, we have the following well-known fact which will be useful in our later calculation.

\begin{lem} \label{fact-sl}
$\Hom_{\U(\sl_2(\K))}(\sl_2(\K)\ot \sl_2(\K),A)\cong \K^2$. More precisely, for any $\vp\in \Hom_{\U(\sl_2(\K))}(\sl_2(\K)\ot \sl_2(\K), A)$,
 there exist unique $\la, \mu\in \K$, such that
$\vp|_{V_5}=0$,
$\vp(h\ot e-e\ot h)=\la e$, $\vp(e\ot f-f\ot e)=\dfrac{\la}{2}h$, $\vp(f\ot h-h\ot f)=\la f$,
and
$\vp(2e\ot f+h\ot h+ 2f\ot e)=\mu 1_A$.
\end{lem}

In fact, the image of $\vp$ on the basis of $\sl_2(\K)\ot \sl_2(\K)$ is shown in the following table.

\begin{center}
{\footnotesize Table 1. The image of $\vp$ on the basis of $\sl_2(\K)\ot \sl_2(\K)$}
\begin{tabular}{|c|c|c|c|}
\hline
$\vp(-\ot-)$ & $e$ & $f$ & $h$\\ \hline
   $e$ &  0 & $\frac{\mu}{6}1_A+\frac{\la}{4}h$ &   $-\frac{\la}{2}e$\\ \hline
   $f$ & $\frac{\mu}{6}1_A-\frac{\la}{4}h$ & 0 & $\frac{\la}{2}f$ \\ \hline
   $h$  & $\frac{\la}{2}e$ & $-\frac{\la}{2}f$ & $\frac{\mu}{3} 1_A$\\ \hline
   \end{tabular}
   \end{center}

To compute the FGV-Poisson cohomology of $A$, we also need  the following facts.

\begin{lem}\label{facts-M2} Keep the above notation. Then

\begin{enumerate}
\item[(1)] $\Ext_{\U(A)}^2(\K,A)=0$,

\item[(2)] $\Ext_{\U(A)}^1(\K,A)$ $\cong \K$,

\item[(3)] $\Hom_\Q(\Omega^2(A),A)\cong \K^2$.
\end{enumerate}
\end{lem}

\begin{proof} The proofs of (1) and (2) are just some routine calculations and omitted here.
We only prove part (3), which needs some technical argument.

By definition, $\Omega^2(A)=A^4/I$, where $I$ is
the submodule of the $\Q$-module $A^4$ generated by
\[\{a\ot b\ot c\ot 1_A-1\ot ab\ot c\ot 1_A+1\ot a\ot bc\ot 1_A-
1\ot a\ot b \ot c\}.\]
Therefore, we know that
\[\Hom_\Q(\Omega^2(A),A)
\cong \{f \in \Hom_{\U(A)}(A^2,A)\mid f\ \mbox{{\rm satisfies\  $(\ast)$}} \},\]
where $(\ast)$ means the equation
\[af(b,c)-f(ab,c)+f(a,bc)-f(a,b)c=0,\eqno{(\ast)}\]
or equivalently, $f$ is a 2-cocycle in the Hochschild complex.

Suppose that $\vp$ is a $\K$-bilinear map satisfying $(\ast)$ in $\Hom_{\U(A)}(A^2,A)$. Then we have $\vp(1_A, 1_A) \in Z(A)$ and
hence $\vp(1_A, 1_A) =\nu 1_A$ for some $\nu \in \K$,
the reason is that $\vp$ is a Lie module homomorphism, that is
\[\{a, \vp(1_A\ot 1_A)\}=\vp(\{a,1_A\ot 1_A\})=0\]
holds for all $a\in A$. By applying $(\ast)$, we therefore obtain that for all $x\in A$,
\[\vp(x\ot 1_A)=\vp(1_A\ot x)=\nu x.\]

On the other hand, each Lie module can be viewed as a module over $\mathfrak{s}\mathfrak{l}_2(\K)$
and \[\Hom_{\U(A)}(A^2,A)=\Hom_{\U(\mathfrak{s}\mathfrak{l}_2(\K))}(A^2,A).\]
Now $\vp$ is determined by $\nu$ and its restriction to $\sl_2(\K)\ot \sl_2(\K)$, and the latter one
is uniquely given by some $\lambda$ and $\mu$ as shown in Lemma \ref{fact-sl}.
By Table 1 and the equation $(\ast)$, we show that
$\mu=3\la$. Therefore $\vp$ is determined by $\nu$ and $\lambda$ as shown in the following table.
\begin{center}
{\footnotesize Table 2. The image of $\vp$ on the $\K$-basis of $A\ot A$}
\begin{tabular}{|c|c|c|c|c|}
\hline
$\vp(-\ot-)$& $1_A$      & $e$                      & $f$                     & $h$\\ \hline
   $1_A$       &  $\nu 1_A$  &$\nu e$                  & $\nu f$                 & $\nu h$ \\ \hline
   $e$         &  $\nu e$   & $0$   & $\dfrac{\la}{2} 1_A+\dfrac{\la}{4}h$ & $-\dfrac{\la}{2} e$\\ \hline
   $f$         &  $\nu f$   & $\dfrac{\la}{2} 1_A-\dfrac{\la}{4}h$  & $0$   & $\dfrac{\la}{2} f$ \\ \hline
   $h$         &  $\nu h$ & $\dfrac{\la}{2} e$                  & $-\dfrac{\la}{2} f$                & $\la$ $1_A$ \\ \hline
   \end{tabular}
   \end{center}
Conversely, any $\nu, \lambda\in \K$ uniquely give to an element $\Phi_{\nu,\la}$ in $\Hom_{
\U(A)}(A^2, A)$ which satisfies $(\ast)$, and hence an element in $\Hom_\Q(\Omega^2(A),A)$. The proof is completed.
\end{proof}

As a vector space, $\mathbb{M}_2(\K)$ is of dimension 4, and hence there are many choices of associative multiplication on it, among which
are two extreme cases. One is given by the matrix product, and the other
one is trivial
An interesting observation is that
these two products are essentially the only cases to make the general linear Lie algebra $\mathbb{M}_2(\K)$ a Poisson algebra.

\begin{cor}
Let $(\mathbb{M}_2(\K), \circ, [-,-])$ be a Poisson algebra,
where $(\mathbb{M}_2(\K), [-,-])$ is the general linear Lie algebra.
Then as an associative algebra, $(\mathbb{M}_2(\K), \circ)$ is either isomorphic to
the matrix algebra, or to the trivial one.
\end{cor}
\begin{proof} For simplicity, we set $P=\mathbb{M}_2(\K)$ and $m(a,b)=a\circ b$ for $a, b\in P$. The Leibniz rule implies that
 $m$ is a homomorphism of Lie modules from $P^2$ to $P$.
From the proof of (4) in Lemma \ref{facts-M2} and Table 1, we know
$4\mu=3\la^2$ since $m(a,m(b,c))=m(m(a,b),c)$. When $\la=0$, the associative algebra $(\mathbb{M}_2(\K), \circ)$
is the trivial associative algebra, and when $\la\neq 0$, it is isomorphic to
the matrix algebra.
\end{proof}

By applying the long exact sequence in Theorem \ref{l.e.s. of Poisson modules} and Lemma \ref{facts-M2}, the Poisson
cohomology groups of $A$ of lower degrees are calculated as follows.

\begin{prop} Keep the above notation. Then
\begin{enumerate}
\item[(1)]  $\HP^0(A)\cong \K$;
\item[(2)] $\HP^1(A)=0$;
\item[(3)] $\HP^2(A)\cong \K$.
\end{enumerate}
\end{prop}

More precisely, $\overline{(\Phi_{0,2},0)}$ gives a basis of $\HP^2(A)$.
In fact, by construction as in Table 2, $\Phi_{0,2}$ is a Lie module homomorphism and satisfies the condition $(\ast)$, which implies that
$(\Phi_{0,2},0)\in \Hom(A^2\oplus \wedge^2, A)$ is a Poisson 2-cocycle of $A$, and its
corresponding cohomology class $\overline{(\Phi_{0,2},0)}$ gives an element in $\HP^2(A)$.

For given $s\in\K$, we may define a $\K[[t]]$-Poisson algebra structure on $A[[t]]$ by setting $l_i=0$ for all $i\ge 0$,
 and $m_t$ to be given as in the following table.
\begin{center}
{\footnotesize Table 3. The image of $m_t$ on the basis of $A\ot A$}
\begin{tabular}{|c|c|c|c|c|}
\hline
$m_t(-,-)$ & $1_A$ & $e$ & $f$ & $h$\\ \hline
   $1_A$   & $1_A$ & $e$ & $f$ & $h$ \\ \hline
   $e$     & $e$ &  0 & $\frac{1}{2}(1-t)h+\frac{1}{2}(1-ts)^21_A$ &   $-e-tse$\\ \hline
   $f$ & $f$ & $-\frac{1}{2}(1-t)h+\frac{1}{2}(1-ts)^21_A$ & 0 & $f+tsf$ \\ \hline
   $h$ & $h$ & $e-tse$ & $-f-tsf$ & $(1-ts)^2$\\ \hline
   \end{tabular}
   \end{center}
It is direct to show that $(A[[t]], m_t, l_t)$ is a formal Poisson deformation
 of $A$ ``lifting" the Poisson $2$-cocycle $s(\Phi_{0,2},0) = (\Phi_{0,2s},0)$, say $(m_1,l_1)= (\Phi_{0,2s},0)$.
 Moreover, for any Poisson 2-cocycle $\eta$, there exists some $s\in\K$ such that
 $\bar\eta = \overline{(\Phi_{0,2s},0)}$ in $\HP^2(A)$, a standard argument shows the
existence of a formal deformation $(A[[t]], m_t', l_t')$ lifting $\eta$ which is equivalent to $(A[[t]], m_t, l_t)$.

\begin{rmk} It is worth mentioning that for the given
Poisson 2-cocyle $(\Phi_{0,2s},0)$, we only give one formal deformation.
It is not known yet whether all formal deformations
lifting it are equivalent.

By the long exact sequence $(\ref{les.pcoh})$ in Theorem \ref{l.e.s. of Poisson modules} and some direct calculations,
we obtain that $\HP^3(A)\neq 0$. Therefore, the above example also tells that
it is still possible to have formal deformations even though ``obstructions" exist.
\end{rmk}

\noindent{\textbf{Acknowledgement}}  The work is supported by the Natural Science Foundation of China (Grant No. 11401001, 11431010 and 11571329).

\vspace{1cm}

{\footnotesize
\noindent Yan-Hong Bao \\
School of Mathematical Sciences, Anhui University, Hefei, China, 230601\\
E-mail address: baoyh@ahu.edu.cn

\vspace{2mm}
\noindent Yu Ye\\
School of Mathematical Sciences, University of Sciences and Technology of China,
Hefei, 230026, China \\
Wu Wen-Tsun Key Laboratory of Mathematics, USTC, Chinese Academy of Sciences,
Hefei, 230026, Anhui, China\\
E-mail address: yeyu@ustc.edu.cn
}

\end{document}